\documentclass[a4paper,11pt]{amsart}
\usepackage[top=30 mm,bottom=30 mm,right=25 mm,left=25 mm]{geometry}

\usepackage{graphicx, mathrsfs,amsmath}
\usepackage[all,poly]{xy}
\usepackage{amscd,amsmath,amsthm,amsfonts,latexsym,amssymb, mathrsfs,stmaryrd,graphics,pst-all,amscd,tkz-euclide,yfonts,mathtools,tikz,graphicx,wrapfig, amsaddr}
\usepackage{amsfonts}
\usepackage{amssymb}
\usepackage{amsmath}
\usepackage{color}
\usepackage[colorlinks]{hyperref}
\usepackage{enumerate}
\scrollmode    
\theoremstyle{plain}
\newtheorem{theorem}{Theorem}[section]
\newtheorem{lemma}[theorem]{Lemma}
\newtheorem{corollary}[theorem]{Corollary}
\newtheorem{proposition}[theorem]{Proposition}

\theoremstyle{definition}

\newtheorem{example}[theorem]{Example}

\newtheorem{remark}[theorem]{Remark}

\numberwithin{equation}{section}

\newcommand{\forget}[1]{}
\begin{document}
	\emergencystretch 3em
	\begin{abstract} We define the notion of dextral symmetric algebras (not necessarily associative), motivated by the idea of symmetric rings. We derive a complete classification of dextral symmetric algebras of Leavitt path algebras, and right Leibniz algebras up to dimension $4$. We also obtain that a finite-dimensional dextral symmetric right Leibniz algebra is solvable if and only if it satisfies a weaker notion of nilpotency.
	\end{abstract}
	\title[On Dextral Symmetric Algebra]{on dextral symmetric algebras}
	\author[Dutta and Bynnud]{Dimpy M. Dutta \and Shanborlang Bynnud}
	\address{Department of Mathematics, North-Eastern Hill University, Shillong, India.}
	\email{dimpymdutta@gmail.com, shanbynnud@gmail.com (Corresponding author)}
	\maketitle
	\section{Introduction}
	Symmetric rings, introduced by Lambek in \cite{lambek} as a subclass of the class of associative rings with unity, undoubtedly play an important role in the domain of ring and module theory. Along with the extensive studies that have been conducted on symmetric rings, algebraists have also introduced many generalizations of these rings (see, for example, \cite{jung, greg, kafkas, huh}). This motivated us to extend this notion to various algebras (not necessarily associative) by introducing the notion of {\em dextral symmetric} algebras.
	
	Recall that an ideal $I$ of an associative ring $R$ is called {\it symmetric} if for any $x,y,z \in R$, $xyz \in I$ implies $xzy \in I$. Lambek called an associative ring with unity {\it symmetric} if the zero ideal is symmetric. We generalize the definition of symmetricity of an ideal of a ring to an ideal $I$ of an algebra $A$ by calling it a {\it symmetric ideal } if it has the property that if  $a(bc) \in I$, then $b(ac ) \in I$, for any $a,b,c \in A$. We call an algebra $A$ (not necessarily associative) a {\em dextral symmetric} algebra if the zero ideal of $A$ is symmetric. Further, an ideal of an algebra is called a {\it dextral symmetric ideal} if it is dextral symmetric as a subalgebra. It is worth mentioning here that in associative algebras with unity, both the symmetric and dextral symmetric conditions are equivalent.
	
	The results on symmetric rings intrigued us to explore whether those results hold in a broader context as well. For example, in a symmetric ring $R$, we have $r_1r_2\dots r_n=0$ implies $ r_{\sigma (1)}r_{\sigma (2)} \dots r_{\sigma (n)}=0$, where $\sigma $ is any permutation of the set $\{1,2,\dots ,n \}$, $r_i \in R$ and $n \in \mathbb N$ \cite[Lemma 1.2]{jung}. Although this result may not be readily true in the case of a non-associative dextral symmetric algebra, in Proposition \ref{sym} and Theorem \ref{sym2}, we observe a similar result for dextral symmetric Leibniz algebras. This further prompted us to introduce another subclass of the class of right Leibniz algebras called the class of left nilpotent right Leibniz algebra. We show that there is a close connection between solvability, dextral symmetricity, and left nilpotency within the class of right Leibniz algebras. 
	
	Another observation is that in the class of anti-commutative algebras, dextral symmetric algebras are, in fact, equivalent to the CB-algebras introduced by Saha and Towers in \cite{ripan}. Recall here that an algebra $A$ is called anti-commutative if for each $x \in A$, $x^2=0$. This also implies that in such an algebra $A$, all $x,y \in A$ satisfy $xy=-yx$. We remind here that an algebra $A$ is called a CB-algebra (as in \cite{ripan}) if whenever $xy=0$ for any $x,y \in A$, we have $(xz)y=0$ for all $z\in A$. In anti-commutative algebras, the CB-algebras display an interesting property: namely, that every centralizer of these algebras is an ideal (see \cite{ripan}).
	
	In Section \ref{sec1}, we document some examples of dextral symmetric algebras in the context of (anti)commutative and (anti)associative algebras. As mentioned in the previous paragraph, we show that in anti-commutative algebras, the notions of dextral symmetric algebras, anti-associative algebras, and CB-algebras are equivalent. In Section \ref{sec2}, we consider the class of Leavitt path algebras and determine the graphs over which a Leavitt path algebra over a commutative ring with unity will be dextral symmetric. This leads to the complete classification of dextral symmetric Leavitt path algebras. Note that a Leavitt path algebra of a graph over a commutative ring with unity is associative. Section \ref{sec3} focuses on results concerning right Leibniz algebras (which are non-commutative, non-associative algebras) that are dextral symmetric. We also introduce a weaker form of nilpotency for a right Leibniz algebra, namely {\em left nilpotency}, and establish an equivalency between a solvable right Leibniz algebra and a left nilpotent right Leibniz algebra through dextral symmetricity. Finally, in Section \ref{classification}, we classify the solvable dextral symmetric right Leibniz algebras up to dimension 4.
	
	\section{Dextral symmetric algebras}\label{sec1}
	In this section, we state some basic properties of dextral symmetric algebras and provide examples based on the commutativity and associativity of the algebra. We also show that in the class of anti-commutative algebras, the classes of dextral symmetric, anti-associative, and CB algebras coincide.  
	\begin{remark}\label{basic properties of dextral symmetric algebras}
		Every subalgebra of a dextral symmetric algebra is dextral symmetric. Moreover, dextral symmetric algebras are preserved under finite direct sums. Thus, every algebra has a largest dextral symmetric ideal. A homomorphic image of a dextral symmetric algebra under an injective homomorphism is also dextral symmetric.
	\end{remark}
	
	It is natural to ask whether the quotient algebra of a dextral symmetric algebra is also dextral symmetric or not. To answer that, let us first look at an example of a class of dextral symmetric algebras.
	
	\begin{example}\label{free}
		Any free algebra $\mathbb F\langle x_1, \dots , x_n\rangle$ (not necessarily associative) over a field $\mathbb F$ is dextral symmetric. Because in a free algebra over a field, the product of any two elements is zero if and only if at least one of the elements is zero.
	\end{example}
	
	The above example helps us in answering the previously asked question in the form of the following example.
	
	\begin{example}\label{notdextral}
		The quotient algebra of a dextral symmetric algebra may not be dextral symmetric. For instance, consider an associative free algebra $A$ over a field $\mathbb F$ generated by three elements $x,y,z$. Take $I$ to be an ideal of $A$ generated by the product $x(yz)=xyz$. Then, $A/I$ is not dextral symmetric as $x(yz)=\bar{0}$ in $A/I$ however, $y(xz)=yxz \neq \bar{0}$ in $A/I$.
	\end{example}
	However, we have the following proposition. The proof is immediate from the definitions.
	\begin{proposition}
		Let $A$ be an algebra and let $I$ be its ideal. Then, $A/I$ is a dextral symmetric algebra if and only if $I$ is symmetric.
	\end{proposition}
	
	\begin{remark}
		The injective hypothesis of remark \ref{basic properties of dextral symmetric algebras} cannot be relaxed, as can be seen by taking the natural projection from the free algebra $A$ of Example \ref{notdextral} onto the algebra $A/I$.
	\end{remark}
	\begin{proposition}\label{Comass}
		An algebra $A$ is dextral symmetric if it has any of the following four properties:
		\begin{enumerate}[{\rm (i)}]
			\item Commutative and associative.
			\item Commutative and anti-associative.
			\item Anti-commutative and associative.
			\item Anti-commutative and anti-associative.
		\end{enumerate}
	\end{proposition}
	\begin{proof}
		We only give the proof for ${\rm (i)}$ as the rest are similar. Let $x,y,z \in A$ such that $x(yz)=0$. Then
		\[ y(xz)=(yx)z=(xy)z=x(yz)=0.\]
	\end{proof}
	
	\begin{proposition}
		For an anti-commutative algebra $A$, the following are equivalent:
		\begin{enumerate}[{\rm (1)}]
			\item $A$ is a dextral symmetric algebra.
			\item $A$ is anti-associative.
			\item $A$ is a CB-algebra.
		\end{enumerate}
	\end{proposition}
	\begin{proof}
		$(1) \implies (3)$. Let $A$ be a dextral symmetric algebra and $x,y \in A$ such that $xy=0$. Then for all $z \in A$,
		\[z(xy)=0 \implies x(yz)=-x(zy)=0 \implies (xz)y=-y(xz)=0.\]
		$(3) \implies (2)$ follows from \cite[Theorem~3.7]{ripan}, while $(2) \implies (1)$ follows from (iv) of Proposition \ref{Comass}.
	\end{proof}
	
	Proposition \ref{Comass} shows that the study of dextral symmetric algebras becomes relevant when we consider the class of algebras that lacks at least one of the following properties: associative, anti-associative, commutative or anti-commutative. Therefore, in the following two sections, we chose to study the property of dextral symmetricity in two classes of algebras, one of which is associative but not commutative/anti-commutative and the other which is neither (anti) associative nor (anti) commutatative.
	
	\section{Classification of dextral symmetric Leavitt path algebras}\label{sec2}
	In this section we consider a Leavitt path algebra over a unital commutative ring $R$. We first characterize the graph associated to a Leavitt path algebra when it is dextral symmetric. We begin by mentioning some basic definitions concerning this class of algebras that are helpful to us in establishing our aim. Finally, we provide a complete classification of dextral symmetric Leavitt path algebras with coefficients taken from a unital commutative ring.
	
	By referring to a graph in this paper, we mean a directed graph $E:=(E^0, E^1, r,s),$ where $E^0$ and $E^1$ are countable sets of vertices and edges respectively. The map $r,s: E^1 \rightarrow E^0$ carries each edge to its range and source, respectively.
	
	A loop is an edge $e \in E^1$ such that $s(e)=r(e)$. A path in a graph $E$ of length $n$ is a sequence of $n$ edges given by $\beta:=e_1e_2 \dots e_n$, where $r(e_i)=s(e_{i+1})$ for $i \in \{1, \dots n-1 \}.$ The vertices in $E^0$ are considered to be paths of length $0$. We denote the set of all finite paths by $E^*:= \cup _{n=0}^{\infty} E^n$, where $E^n$ denotes the set of all paths of length $n$. 
	
	For an edge $e \in E^1$ with $s(e)=u,~ r(e)=v$, where $u,v \in E^0$, we have a corresponding ghost edge $e^*$ such that $s(e^*)=v, ~r(e^*)=u.$ We denote the collection of all such ghost edges by $(E^1)^*$. 
	
	Now we give the formal definition of a Leavitt path algebra over a unital commutative ring $R$. 
	
	The {\em Leavitt path algebra} of an arbitrary directed graph $E$ over a unital commutative ring $R$, denoted by $L_R(E)$ is an $R$-algebra generated by a set of pairwise orthogonal idempotent elements $\{v ~ \vert ~ v \in E^0\}$ and a set $\{e,e^* ~ \vert ~ e\in E^*\}$ of edges and ghost edges of $E$ satisfying the following conditions :
	\begin{enumerate}[(1)]
		\item $s(e)e = e = er(e)$ for all $e \in E^{1}$.
		\item $r(e)e^{*} = e^{*} = e^{*}s(e)$ for all $e \in E^{1}$.
		\item (The CK-1 relations) For all $e, f \in E^{1},\ e^{*}e = r(e)$ and    $e^{*}f = 0$ if $e \neq f$.
		\item (The CK-2 relations) For every regular vertex $v \in E^{0}$, 
		\[v = \sum\limits_{e \in E^{1},\ s(e)=v} ee^{*}.\]
	\end{enumerate}
	
	\begin{remark}\label{nonzero}
		\begin{enumerate}[(1)]
			\item Note that all the elements of the set $\{v,e,e^* ~ \vert ~ v \in E^0, ~ e \in E^1\}$ are non-zero in $L_R(E)$ \cite[Proposition~3.4]{tomforde}. 
			\item Every element $x \in L_R(E)$ can be expressed as $x =\sum_{j=1}^{m}l_{j}\alpha_{j}\gamma_{j}^{*} $, where $l_{i} \in R$, $\alpha_{j}$, $\gamma_{j} \in E^*$ and $m$ is a suitable integer.
		\end{enumerate}   
	\end{remark}
	
	\begin{lemma}\label{sourcerange}
		In a dextral symmetric Leavitt path algebra the associated graph has no edge with distinct source and range.
	\end{lemma}
	\begin{proof}
		Let us prove the lemma by contradiction. Consider a dextral symmetric Leavitt path algebra $L_R(E)$. Suppose, that the graph $E$ contains an edge $e\in E^1$ such that $s(e)\neq r(e)$. Then we would have $es(e)r(e)=0$. However, $s(e)er(e)=e \neq 0$ (as mentioned in Remark \ref{nonzero}). This contradicts our assumption that $L_R(E)$ is dextral symmetric. Thus, $E$ contains no edge with distinct source and range.
	\end{proof}
	
	\begin{lemma}\label{leavlem1}
		In a dextral symmetric Leavitt path algebra, the associated graph contains no vertex emitting more than one loop.
	\end{lemma}
	\begin{proof}
		Let $L_R(E)$ be a dextral symmetric Leavitt path algebra such that the associated graph $E$ contains a vertex $v$ which emits at least two distinct loops say $g$, $f$. Then by CK-1 and Remark \ref{nonzero}, $g^*fv=g^*f=0$, but $fg^*v =fg^*\neq 0.$ Which is a contradiction.
	\end{proof}
	
	\begin{lemma}\label{leavlem2}
		The Leavitt path algebra $L_R(E)$ of a graph $E$ with a single vertex and a single loop is dextral symmetric.
	\end{lemma}
	\begin{proof}
		The proof is straightforward. The Leavitt path algebra of such a graph is isomorphic to the Laurent polynomial algebra $R[x,x^{-1}]$ for an indeterminate $x$, as shown in \cite[Proposition~1.3.4]{Abrams}. Since $R$ is commutative, $R[x,x^{-1}]$ is commutative. Therefore, by part (i) of Proposition \ref{Comass}, $R[x,x^{-1}]$ is dextral symmetric.
	\end{proof}
	
	\begin{remark}\label{leavrem}
		With a similar argument as in Lemma \ref{leavlem2}, one can observe that a Leavitt path algebra $L_R(E)$ of a graph $E$ with a single vertex is always dextral symmetric.
	\end{remark}
	
	As a consequence of the above results, we have the following Lemma.
	
	\begin{lemma}
		A Leavitt path algebra $L_R(E)$ is dextral symmetric iff each component of the graph $E$ is either of the following two types.
		
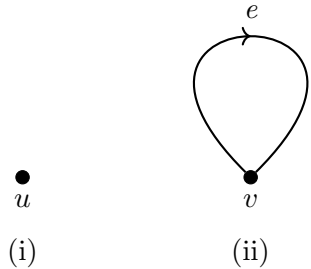
\begin{center}
			\begin{tikzpicture}
				\node[draw, circle, fill=black, inner sep=1pt, minimum size=5pt,          label=below: $u$] (u) at (0,0) {};
				\node at (0,-1) {\rm (i)};
				
				\node[draw, circle, fill=black, inner sep=1pt, minimum size=5pt, label=below: $v$] (v) at (3,0) {};
				\draw[thick, postaction={decorate,decoration={markings,mark=at position 0.5 with {\arrow{>}}}}] (v) to[out=135,in=45,looseness=65] (v);
				\node at (3,-1) {\rm (ii)};
				\node at (3,2.2) {e};
			\end{tikzpicture}
		\end{center}
	\end{lemma}
	\begin{proof}
		The proof follows directly from Lemma \ref{sourcerange}, Lemma \ref{leavlem1}, Lemma \ref{leavlem2} and Remark \ref{leavrem}.
	\end{proof}
	
	This finally leads us to give the complete classification of a dextral symmetric Leavitt path algebra, $L_R(E)$.
	\begin{theorem}
		A dextral symmetric Leavitt path algebra $L_R(E)$ is isomorphic to :
		$$\bigoplus \limits_{I} R \oplus \bigoplus \limits_J R[x,x^{-1}].$$
		for some index sets $I$ and $J$.
	\end{theorem}
	
	\section{Leibniz algebra} \label{sec3}
	We begin this section by recalling the definition of a right Leibniz algebra. A {\em right Leibniz algebra} $\mathscr L$ over a field $F$ is a vector space over the field $F$ equipped with a bilinear product $[ , ]: \mathscr L \times \mathscr L \rightarrow \mathscr L$ called a bracket of $\mathscr L$ satisfying the following (right) Leibniz identity:
	$$[a,[b,c]]=[[a,b],c] - [[a,c],b] ~ \text{for each $a,b,c \in \mathscr L$, \cite{ayupov}}.$$
	The analogous notion of {\em left Leibniz algebra} $\mathscr L$ satisfies the following (left) Leibniz identity:
	$$[a,[b,c]] = [[a,b],c]+[b,[a,c]] ~ \text{for each $a,b,c \in \mathscr L$, \cite{nil}}.$$ 
	We start off our study of dextral symmetric right Leibniz algebras with the next proposition. We show that for every triplet of elements $x,y,z$ in a dextral symmetric right Leibniz algebra, the product of the form $[x, [y,z]]$ is equal to the product (of the same form) obtained by any permutation of $x,y$ and $z$ up to sign. Specifically, if the permutation is taken clockwise the sign is positive otherwise the sign is negative.
	
	\begin{proposition}\label{sym}
		If a right Leibniz algebra $\mathscr L$ is dextral symmetric, then  for all $x,y,z \in \mathscr L$, the following relations hold:
		$$[x,[y,z]]=[y,[z,x]]=[z,[x,y]]=-[x,[z,y]]=-[y,[x,z]]=-[z,[y,x]],$$
	\end{proposition}
	\begin{proof}
		Since $\mathscr L$ is a right Leibniz algebra, we always have $[y,[x+z,x+z]]=0$ for all $x,y,z \in \mathscr L$. Moreover, $\mathscr L$ is dextral symmetric, which implies that $[x+z,[y,x+z]]=0$,  leading to $[x,[y,z]]=-[z,[y,x]]$.
		It follows from the Leibniz identity that $[x,[y,z]]=-[x,[z,y]]$. By replacing $x$ with $y$, $y$ with $z$, and $z$ with $x$, the same argument can be applied to obtain the required result.
	\end{proof}
	
	\begin{theorem}\label{sym2}
		A right Leibniz algebra $\mathscr L$ which is also a left Leibniz algebra is dextral symmetric iff it is anti-associative.
	\end{theorem}
	\begin{proof} 
		A right Leibniz algebra $\mathscr L$ which is also a left Leibniz algebra satisfies
		$$[x,[y,z]]=-[x,[z,y]]=-[[y,z],x]=[[z,y],x],$$
		for all $x,y,z \in \mathscr L.$ Assume $\mathscr L$ to be dextral symmetric. Then by Proposition \ref{sym}, for any $x,y,z \in \mathscr L$ we have the following relations,
		$$[x,[y,z]]=[z,[x,y]]=-[[x,y],z].$$
		Conversely, let $\mathscr L$ be anti-associative. If we have elements $x,y,z\in \mathscr L$ such that $[x,[y,z]]=0$, then we get,
		$$[y,[x,z]]=-[[x,z],y]= [x,[z,y]]=0.$$
	\end{proof}
	
	It is worth mentioning here that not every dextral symmetric right Leibniz algebra is anti-associative. For instance, Example \ref{lnotr} is dextral symmetric but not anti-associative.
	
	Henceforth, unless otherwise mentioned, the term Leibniz algebra will mean a right Leibniz algebra.
	
	Over the years, many studies have been carried out on solvable and right nilpotent Leibniz algebras, \cite{ayupov, omirov, nil, Abdu, rakh, ismail}. In this paper, we introduce another class of Leibniz algebras that lies between the classes of right nilpotent and solvable Leibniz algebras. We introduce the notion of left nilpotency, which is a weaker notion than the usual notion of nilpotency. However, these two notions of nilpotency coincide for a Lie algebra. In order to give the formal definition of this new class of Leibniz algebras, let us mention the following four sequences of two-sided ideals of a given Leibniz algebra $\mathscr L$.
	\begin{enumerate}[(i)]
		\item $\mathscr L^{<1>} = \mathscr L, ~ \mathscr L^{<n+1>}=[\mathscr L^{<n>}, \mathscr L];$
		\medskip
		\item $\mathscr L^{(1)} = \mathscr L, ~ \mathscr L^{(n+1)} = [\mathscr L, \mathscr L^{(n)}]$;
		\medskip
		\item $\mathscr L^{[1]}= \mathscr L, ~ \mathscr L^{[n+1]} = [\mathscr L^{[n]}, \mathscr L^{[n]}];$
		\medskip
		\item $\mathscr L ^1= \mathscr L, ~ \mathscr L^{n+1}= [\mathscr L^1, \mathscr L^n] + [\mathscr L^2, \mathscr L^{n-1}]+ \cdots +[\mathscr L^{n-1}, \mathscr L^2] +[\mathscr L^n, \mathscr L^1].$
	\end{enumerate}
	The algebra $\mathscr L$ is said to be nilpotent if $\mathscr L^n=0$, for some $n \in \mathbb N$, right nilpotent if $\mathscr L^{<k>}=0$, for some $k \in \mathbb N$, and solvable if  there exists $n^{\prime} \in \mathbb N$ such that $\mathscr L^ {[n^\prime]}=0$ \cite{ayupov}.
	
	We call a Leibniz algebra $\mathscr L$ left nilpotent if there exists $k^\prime \in \mathbb N$ such that $\mathscr L^{(k^\prime)}=0$. A similar notion can be defined for left Leibniz algebras. It easily follows that a subalgebra as well as a quotient algebra of a left nilpotent Leibniz algebra is left nilpotent.
	
	It was shown in \cite{ayupov} that a Leibniz algebra is right nilpotent if and only if it is nilpotent. Also, it is well known that a right nilpotent Leibniz algebra is solvable, but not vice versa.
	
	It is easy to see from the Leibniz identity that in a Leibniz algebra, right nilpotency will imply left nilpotency. But Example \ref{lnotr} below shows that the reverse implication may not hold in general. Thus implying that in a Leibniz algebra, right nilpotency and left nilpotency are two different concepts. This distinction comes mainly from the fact that the brackets in a generic Leibniz algebra do not commute.
	
	\begin{example}\label{lnotr}
		Let $\mathscr L$ be the Leibniz algebra of dimension $3$, spanned by $\{x,y,z\}$ with the non-zero product given by
		$$[z,x]=z.$$
		Then $\mathscr L$ is left nilpotent but not right nilpotent.
	\end{example}
	
	\begin{theorem}
		Let $\mathscr L$ be a finite dimensional Leibniz algebra over a field of characteristic zero. If $\mathscr L$ is left nilpotent, then it is solvable.
	\end{theorem}
	\begin{proof}
		Assume that $\mathscr L$ is left nilpotent. By the Levi-Malcev theorem for Leibniz algebra \cite[Theorem 1]{barnes}, $\mathscr L$ can be decomposed as $\mathscr L = R+S$, where $R$ is the solvable radical of $\mathscr L$, and $S$ is a subalgebra of $\mathscr L$ which is a semi-simple Lie algebra. We now claim that $S$ must be zero. If not, since $S$ is a subalgebra of $\mathscr L$, it is also left nilpotent and hence a nilpotent Lie algebra. Consequently, $S$ is solvable. This contradicts the fact that there exists no non-zero Lie algebra that is both semi-simple and solvable. Therefore, $S=0$ yielding that $\mathscr L$ is solvable.
	\end{proof}
	However, the following example shows that the converse of the above theorem is not true. 
	
	\begin{example}
		Consider the 3-dimensional Leibniz algebra $\mathscr L^\prime$ generated by $\{x,y,z\}$ with the non-zero products given by:
		$$[x,z]=x; ~ [y,z]=y; ~ [z,y]=-y.$$
		Clearly, $\mathscr L^{\prime}$ is solvable as $\mathscr {L^{\prime}}^{[3]}=0$. But it is not left nilpotent. 
	\end{example}
	
	Thus in the class of finite dimensional Leibniz algebras over a field of characteristic zero, we have the following strict inclusions:
	\[ 
	\begin{Bmatrix}
		\text{ Right nilpotent}
	\end{Bmatrix} \subsetneq \begin{Bmatrix}
		\text{ Left nilpotent}
	\end{Bmatrix} \subsetneq \begin{Bmatrix}
		\text{Solvable } 
	\end{Bmatrix}.
	\]
	In what follows, we establish that in the class of dextral symmetric Leibniz algebras, the notions of solvability and left nilpotency coincide.
	\begin{lemma}\label{L4}
		In a dextral symmetric Leibniz algebra $\mathscr L$, for any elements $x,y,z,w \in \mathscr L,$ the following relation holds:
		$$[[x,y],[z,w]]=[x,[y,[z,w]]].$$
	\end{lemma}
	\begin{proof}
		Let $x,y,z,w$ be any four elements of $\mathscr L$, then by repeatedly applying Proposition \ref{sym} it follows that
		$$[[x,y], [z,w]]=[z,[w,[x,y]]]=[z,[x,[y,w]]]$$
		$$=-[x,[z,[y,w]]]=[x,[y,[z,w]]].$$
	\end{proof}
	
	\begin{corollary}\label{L4cor}
		Let $\mathscr L$ be a dextral symmetric Leibniz algebra. For any two subspaces $A,B$ of $\mathscr L$, and for any integer $n\geq 2$,
		\[[[A,B],\mathscr L^{(n)}]=[A,[B,\mathscr L^{(n)}].\]
		Thus, in particular for any three integers $ m_1,m_2 \geq 1$, and  $n \geq 2$,
		$$[[\mathscr L^{[m_1]}, \mathscr L^{[m_2]}], \mathscr L^{[n]}]=[\mathscr L^{[m_1]},[ \mathscr L^{[m_2]}, \mathscr L^{(n)}]].$$
	\end{corollary}
	
	\begin{remark}\label{solnil} 
		Lemma \ref{L4} shows that for a dextral symmetric Leibniz algebra $\mathscr L,  ~\mathscr L^{[3]}=\mathscr L^{(4)}$. We generalize this observation in the subsequent propositions.
	\end{remark}
	
	\begin{proposition}\label{solnil2}
		Let $\mathscr L$ be a dextral symmetric Leibniz algebra. Then for each $m,n \geq 1$, $$[\mathscr L^{[m]}, \mathscr L^{(n)}]=\mathscr L^{(2^{(m-1)}+n)}.$$
	\end{proposition}
	\begin{proof}
		We proceed by induction on $m$. For $m=1$ and any $n \geq 1$, we have
		$$[\mathscr L^{[1]}, \mathscr L^{(n)}]= \mathscr L^{(n+1)}= \mathscr L^{(2^{(1-1)}+n)}.$$
		Thus it is true for $m=1$. Let us assume the result holds for $m$, i.e.,
		$$[L^{[m]}, \mathscr L^{(n)}]=\mathscr L^{(2^{(m-1)}+n)}.$$
		By Corollary \ref{L4cor} and the induction hypothesis we get:
		\begin{align*}
			[\mathscr L^{[m+1]}, \mathscr L^{(n)}]&=[\mathscr L^{[m]}, [\mathscr L^{[m]},\mathscr L^{(n)}]]
			\\&=[\mathscr L^{[m]}, \mathscr L^{(2^{(m-1)}+n)}]
			\\& = \mathscr L^{(2.2^{(m-1)}+n)}
			\\& = \mathscr L^{(2^{m}+n)}.
		\end{align*}
	\end{proof}
	
	\begin{proposition}\label{solnil3}
		Let $\mathscr L$ be a dextral symmetric Leibniz algebra. Then for any $m \geq 1,$ $$\mathscr L^{[m]}= \mathscr L^{(2^{(m-1)})}.$$
	\end{proposition}
	\begin{proof}
		We prove this proposition by induction on $m$. First of all, it is easy to see that the result is true for $m=1$. Assuming that it is true for $m$, we have:
		$$\mathscr L^{[m]}= \mathscr L^{(2^{(m-1)})}.$$
		Now, from the above assumption we may write $$\mathscr L^{[m+1]}= [\mathscr L^{[m]}, L^{[m]}] = [L^{[m]},  \mathscr L^{(2^{(m-1)})}].$$
		Following Proposition \ref{solnil2}, we can conclude that
		$$\mathscr L^{[m+1]}= \mathscr L^{(2^{(m-1)}+2^{(m-1)})}= \mathscr L^{(2^{m})}.$$
	\end{proof}
	We are now in a position to characterize a solvable dextral symmetric algebra.
	\begin{theorem}\label{dextral symmetricnil}
		A dextral symmetric Leibniz algebra $\mathscr L$ is solvable if and only if it is left nilpotent.
	\end{theorem}
	\begin{proof}
		The fact that solvability will imply left nilpotency is clear from Proposition \ref{solnil3}. Also, if $\mathscr L^{(n)}=0$ for some $n \in \mathbb N$, we have $\mathscr L^{(2^n)}=0$. Again, by Proposition \ref{solnil3} $ \mathscr L^{[n-1]}=0.$
	\end{proof}
	\begin{remark}
		\begin{enumerate}[(1)]
			\item It is worth mentioning here that a dextral symmetric solvable and hence left nilpotent right Leibnitz algebra may not be right nilpotent. For instance, the Leibniz algebra in Example \ref{lnotr} is dextral symmetric as well as solvable, whereas it is not right nilpotent.
			\item Since in a right Leibniz algebra which is also a left Leibniz algebra, the concepts of left nilpotency and right nilpotency are the same, it follows from Theorem \ref{dextral symmetricnil} that in a dextral symmetric right Leibniz algebra which is also a left Leibniz algebra all three notions of right nilpotency, left nilpotency and solvability coincide.
		\end{enumerate}
	\end{remark}
	\begin{proposition}\label{leftnil2}
		Let $\mathscr L$ be an $n$-dimensional dextral symmetric Leibniz algebra. Then $\mathscr L^{(n+1)}=0$.
	\end{proposition}
	\begin{proof}
		Let $\{x_1, x_2, \cdots, x_n\}$ be a basis of $\mathscr L$. Then any element of $\mathscr L^{(n+1)}$ is a linear combination of the following product of $n+1$ basis elements:
		$$[x_{i_{1}},[x_{i_{2}},[\cdots ,[x_{i_n},x_{i_{n+1}}]\cdots]]],$$
		where $i_1, \ldots , i_{n+1}\in \{1,\ldots ,n\}$. It follows that there exists at least one pair of distinct indices $i_{j}$ and $i_{k}$ such that $x_{i_j}=x_{i_k}$. Since $\mathscr L$ is dextral symmetric, by Proposition \ref{sym} and the Leibniz identity,
		$$[x_{i_{1}},[x_{i_{2}},[\cdots ,[x_{i_n},x_{i_{n+1}}]\cdots ]]]=\pm [x_{i_{1}},[x_{i_2}, [\cdots ,[x_{i_{n+1}},[x_{i_{j}},x_{i_{k}}]]\cdots]]]=0$$
		Hence $\mathscr L^{(n+1)}=0$.
	\end{proof}
	\begin{corollary}\label{clss}
		Any finite dimensional dextral symmetric Leibniz algebra is left nilpotent and hence solvable.
	\end{corollary}
	
	\section{Classification of complex dextral symmetric Leibniz algebras up to dimension 4}\label{classification}
	In this section, we consider the Leibniz algebras over the field of complex numbers. It follows from Corollary \ref{clss} that to classify 44-dimensional dextral symmetric Leibniz algebra, it is enough to classify them based on the classification of 44-dimensional solvable Leibniz algebras in \cite{omirov}, \cite{ismail} and \cite{elisa}.
	
	In what follows, we mention only the non-zero brackets of basis elements of a Leibniz algebra.
	\begin{proposition}
		If $\mathscr L$ is a $2$-dimensional right nilpotent Liebniz algebra, then $\mathscr L$ is dextral symmetric.
	\end{proposition}
	\begin{proof}
		If $\mathscr L$ is abelian, then it is clearly dextral symmetric. If it is non abelian, it will be isomorphic to the following algebra:
		$$\gamma_1:~ [x,x]=y,$$
		where $\{x,y\}$ is the basis of $\mathscr L$. Notice that in $\gamma_1$, the element $[a,[b,c]]$ is zero for all $a,b,c \in \mathscr L.$ Hence, $\mathscr L$ is trivially dextral symmetric.
	\end{proof}
	
	\begin{proposition}
		Every $3$-dimensional right nilpotent Leibniz algebra is dextral symmetric.
	\end{proposition}
	\begin{proof}
		Let $\mathscr L$ be a $3$-dimensional right nilpotent Leibniz algebra. The result is true if $\mathscr L$ is abelian. If $\mathscr L$ is non abelian, and if $\{x,y,z\}$ is its basis, it is isomorphic to one of the following pairwise non-isomorphic algebras:
		\begin{align*}
			\mu_1 :~&[x,x]=z~;
			\\\mu_2:~&[x,y]=z, [y,x]=-z~;
			\\\mu_3(\alpha):~ &[x,x]=z,~ [y,y]=\alpha z,~ [x,y]=z~;
			\\ \mu_4:~ &[y,x]=z, ~ [x,y]=z ~;
			\\ \mu_5:~ &[x,x]=y,~ [y,x]=z.
		\end{align*}
		In each of the above cases, $\mathscr L^{(3)}=0.$ Therefore, $\mathscr L$ is dextral symmetric.
	\end{proof}
	
	We assume that a 4-dimensional right Leibinz algebra is spanned by $\{x,y,z,w\}$.
	
	\begin{theorem}
		There are exactly $18$ isomorphism classes of $4$-dimensional right nilpotent dextral symmetric Leibniz algebras.
	\end{theorem}
	\begin{proof}
		According to \cite{omirov} and \cite{ismail}, the isomorphism classes of right nilpotent Leibniz algebras are given by the following algebras.   
		\begin{align*}
			\mathscr N_1 &:~ [x,x]=y, ~ [y,x]=z, ~[z,x]=w ~;
			\\ \mathscr N_2 &:~ [x,x]=z,~[x,y]=w,~[y,x]=z,~[z,x]=w~;
			\\ \mathscr N_3 &:~ [x,x]=z,~[y,x]=z,~[z,x]=w~;
			\\ \mathscr N_4(\alpha) &:~ [x,x]=z,~[x,y]= \alpha w,~[y,x]=z,~[y,y]=w,~[z,x]=w,~ \alpha \in \lbrace 0,1 \rbrace ~;
			\\ \mathscr N_5 &:~[x,x]=z,~[x,y]=w,~[z,x]=w~;
			\\ \mathscr N_6 &:~[x,x]=z,~[y,y]=w,~[z,x]=w~;
			\\ \mathscr N_7 &:~[x,x]=w,~[y,x]=z,~[z,x]=w,~[x,y]=-z,~[x,z]=-w~;
			\\ \mathscr N_8 &:~[x,x]=w,~[y,x]=z,~[z,x]=w,~[x,y]=-z+w,~[x,z]=-w~;
			\\ \mathscr N_9 &:~[x,x]=w,~[y,x]=z,~[y,y]=w,~[z,x]=w,~[x,y]=-z+2w,~[x,z]=-w~;
			\\ \mathscr N_{10} &:~[x,x]=w,~[y,x]=z,~[y,y]=w,~[z,x]=w,~[x,y]=-z,~[x,z]=-w~;
			\\ \mathscr N_{11} &:~[x,x]=w,~[x,y]=z,~[y,x]=-z,~[y,y]=-2z+w~;
			\\ \mathscr N_{12} &:~[x,y]=z,~[y,x]=w,~[y,y]=-z~;
			\\ \mathscr N_{13}(\alpha) &:~[x,x]=z,~[x,y]=w,~[y,x]=-\alpha z, ~[y,y]=-w,~ \alpha \in \mathbb C~;
			\\ \mathscr N_{14}(\alpha) &:~[x,x]=w,~[x,y]= \alpha w,~[y,x]=- \alpha w,~[y,y]=w,~[z,z]=w,~ \alpha \in \mathbb C~;
			\\\mathscr N_{15} &:~[x,y]=w,~[x,z]=w,~[y,x]=-w,~[y,y]=w,~[z,x]=w~:
			\\ \mathscr N_{16} &:~[x,x]=w,~[x,y]=w,~[y,x]=-w,~[z,z]=w~;
			\\ \mathscr N_{17} &:~[x,y]=z,~[y,x]=w~;
			\\ \mathscr N_{18} &:~[x,y]=z,~[y,x]-z,~[y,y]=w~;
			\\ \mathscr N_{19} &:~[y,x]=w,~[y,y]=z~;
			\\ \mathscr N_{20}(\alpha) &:~[x,y]=w,~[y,x]= \frac{1+\alpha}{1-\alpha}w,~[y,y]=z,~ \alpha \in \mathbb C \setminus \{1\}~;
			\\ \mathscr N_{21}&:~[x,y]=w,~[y,x]=-w,~[z,z]=w ~;
			\\ \mathscr N_{22}&:~[z,x]=w,~[y,z]=w.
		\end{align*}
		Notice that in each of the algebras from $\mathscr N_7$ to $\mathscr N_{10}$ we have $[y,[x,x]]=0$ but $[x,[y,x]] \neq 0$. Whereas, for all the other $18$ isomorphism classes, $\mathscr N_i^{(3)}=0$. 
	\end{proof}
	
	\begin{theorem}\label{2dimnil}
		A $4$-dimensional solvable Leibniz algebra with 2-dimensional nilradical is dextral symmetric if and only if it is isomorphic to the algebra
		$$\mathscr R_1: ~[x,z]=x, ~[y,w]=y.$$
	\end{theorem}
	\begin{proof}
		A $4$-dimensional solvable Leibniz algebra with 2-dimensional nilradical is isomorphic to the following pairwise non-isomorphic algebras:
		\begin{align*}
			&\mathscr R_1: ~[x,z]=x, ~[y,w]=y~;
			\\&\mathscr S_2:~ [x,z]=x, ~ [y,w]=y,~ [z,x]=-x, ~ [w,y]=-y~;
			\\& \mathscr S_3:~ [x,z]=x,~ [y,w]=y, ~[w,y]=-y.
		\end{align*}
		Since in $\mathscr R_1$, $[a,[b,c]]=0$ for each $a,b,c \in \mathscr R_1$, it is clearly dextral symmetric.
		
		We now claim that $\mathscr S_2$ and $\mathscr S_3$ are not dextral symmetric. We will only show that $\mathscr S_2$ is not dextral symmetric, as the proof for $\mathscr S_3$ is similar.
		
		Assume for contradiction that $\mathscr S_2$ is dextral symmetric. Then $[b,[a,b]]=0$ for all $a,b \in \mathscr S_2$. However, $[z,[x,z]]=-x \neq 0$. Thus $\mathscr S_2$ can not be dextral symmetric.
	\end{proof}
	
	\begin{remark}
		Note that in Theorem \ref{2dimnil}, $\mathscr S_2$ and $\mathscr S_3$ are not left nilpotent. As in both the algebras, the element 
		$$\underbrace{[w,[w,[w, \cdots w,[w}_{n-times},y]\cdots]]]= (-1)^ny \neq 0, \text{ for any } n \in \mathbb N.$$
		Hence, $\mathscr S_i^{(n)} \neq 0$ for all $n \in \mathbb N$, where $i\in \{2,3\}.$
	\end{remark}
	
	\begin{proposition}
		There exists no dextral symmetric Leibniz algebra of dimension $4$ with nilradical isomorphic to $\mu_2$.
	\end{proposition}
	\begin{proof}
		Any $4$-dimensional solvable Leibniz algebra with nilradical isomorphic to $\mu_2$ is isomorphic to one of the following algebras which are pairwise non-isomorphic:
		\begin{align*}
			\mathscr L_1(\lambda): ~ &[y,z]=w,
			~[z,y]=-w,
			~[y,x]=y,
			~[z,x]= \lambda z,
			~[w,x]=(1+\lambda)w,
			\\&[x,y]=-y,
			~[x,z]=-\lambda z,
			~[x,w]=(1+\lambda)w.
			\\\mathscr L_2: ~&[y,z]=w,
			~[z,y]=-w,
			~[y,x]=y,
			~[x,y]=-y,
			\\&[z,x]=-z,
			~[x,z]=z,
			~[x,x]=w. 
			\\ \mathscr L_3:~ &[y,z]=w, ~[z,y]=-w, ~[y,x]=y+z,~ [x,y]=-y-z, \\&[z,x]=z,~[x,z]=-z, [w,x]=2w,~ [x,w]=-2w.
		\end{align*}
		Similar to the proof of Theorem \ref{2dimnil}, we can show that in each of the above mentioned algebras, there exists elements $a,b$ such that $[b,[a,b]] \neq 0$. Which yields that the algebras $\mathscr L_1(\lambda)$, $\mathscr L_2$ and $\mathscr L_3$ are not dextral symmetric.
	\end{proof}
	\begin{remark}
		$\mathscr L_1(\lambda)$, $\mathscr L_2$ and $\mathscr L_3$ are not left nilpotent.
	\end{remark}
	
	\begin{proposition}
		A $4$-dimensional solvable Leibniz algebra with nilradical isomorphic to $\mu_3(\alpha)$ is dextral symmetric if and only if it is isomorphic to the algebra
		$$\mathscr R_2: ~ [z,y]=w, ~ [z,x]=z, ~ [w,x]=w.$$
	\end{proposition}
	
	\begin{proposition}
		A $4$-dimensional solvable Leibniz algebra with nilradical isomorphic to $\mu_1$ is dextral symmetric if and only if it is isomorphic to the following three pairwise non-isomorphic algebras:
		\begin{align*}
			\mathscr R_3{(\beta)}: ~&[y,y]=w, ~[z,x]=z, ~[x,y]=w, ~[x,x]=\beta w, ~ w\in \mathbb C~; \\
			\mathscr R_4:~&[y,y]=w, ~[z,x]=z~;
			\\ 
			\mathscr R_5: ~&[y,y]=w, ~[z,x]=z, ~[x,y]=w, ~[y,x]=w.
		\end{align*}
	\end{proposition}
	
	\begin{remark}
		Note that $\mathscr R_3(\beta)$ and $\mathscr R_i$ are left nilpotent for each $i\in \{2,4,5\}$ with index of left nilpotency equal to $3$.
	\end{remark}
	
	\begin{proposition}
		There is no $4$-dimensional dextral symmetric solvable Leibniz algebra with nilradical isomorphic to $\mu_4$ or $\mu_5$.
	\end{proposition}
	
	Following the discussions above, we have a complete characterization of a $4$-dimensional solvable dextral symmetric Leibniz algebra.
	
	\begin{theorem}
		A $4$-dimensional Leibniz algebra $\mathscr L$ is dextral symmetric if and only if $\mathscr L^{(3)}=0.$
	\end{theorem}
	
	For every $n\in \mathbb{N}$, there exists an $n$-dimensional dextral symmetric Leibniz algebra $\mathscr L$ over any field satisfying $\mathscr L^{(3)}=0.$
	
	\begin{example}\cite[Example 2.2]{Towers}
		Let $\mathscr L$ be a Leibniz algebra spanned by basis elements $\{x_1,\ldots, x_{n-1},x_n\}$ with non-zero products given by
		\[[x_i,x_n]=x_i \text{ ~for each }i\in \{1,\ldots, n-1\}.\]
		Then, clearly $\mathscr L^{(3)}=0.$
	\end{example}
	
	We bring this paper to a close with an example of a $7$-dimensional Lie algebra to show that not every finite-dimensional dextral symmetric Leibniz algebra $\mathscr L$ will satisfy $\mathscr L^{(3)}=0.$
	
	\begin{example} \cite[Example 4.6]{ripan} 
		Let $\mathscr L$ be a $7$-dimensional Lie algebra over a field of characteristic 3. Consider the set $\{ x_1,x_2,x_3,x_4,x_5,x_6,x_7\}$ as the basis of $\mathscr L$ with multiplication given by
		\begin{align*}
			&[x_1,x_2]=x_4, \hspace{1.5cm} [x_1,x_3]=x_5, \hspace{1.5cm} [x_2,x_3]=x_6,
			\\&[x_1,x_6]=x_7,\hspace{1.5cm} [x_2,x_5]=-x_7, \hspace{1.2cm}  [x_3,x_4]=x_7.
		\end{align*}
		It can be verified that $\mathscr L$ is a dextral symmetric Lie algebra. However, $[x_1,[x_2,x_3]]=x_7 \neq 0.$
	\end{example}
	
	\section*{Acknowledgements}
	We express our heartfelt gratitude to Dr. Ardeline Mary Buhphang (NEHU, Shillong), Rishabh Goswami (NEHU, Shillong) and Dr. Ripan Saha (Raiganj University) for their valuable suggestions while preparing this manuscript. The first author would like to thank Prof. Samuel Lopes (University of Porto) for his insightful conversations during CIMPA School-2024, held in Philippines. We also thank Ivan Kaygorodov's (Federal University of ABC, Brazil) for identifying an error in our paper and bringing it into our attention.

\end{document}